\documentclass[12pt]{article}
\usepackage{amsmath, amsfonts, amssymb, amsthm} 
\usepackage[colorlinks = true, linkcolor = blue]{hyperref}
\usepackage[nottoc]{tocbibind}
\newtheorem{theorem}{Theorem}[section]
\newtheorem{proposition}[theorem]{Proposition}
\newtheorem{definition}[theorem]{Definition}
\newtheorem{remark}[theorem]{Remark}

\newtheorem{example}[theorem]{Example}

\newtheorem{corollary}[theorem]{Corollary}
\newtheorem{lemma}[theorem]{Lemma}

\def\R{\mathbb R} \def\Z{\mathbb Z} \def\C{\mathbb C} 

\def\N{\mathbb N}
\def\C{{\mathbb C}}

\def\Ad{\hbox{\rm Ad}}

\def\GL{\mathrm {GL}}
\def\SL{\mathrm {SL}}

\def\<{\,<\!}
\def\>{\!>\,}
\usepackage{graphics}
\usepackage{epsf}
\usepackage{epsfig}

\def\R{\mathbb{R}}
\def\C{\mathbb{C}}
\def\F{\mathbb{F}}
\def\N{\mathbb{N}}

\def\GL{{\rm GL}}
\def \id {{\rm id}}

\def \Ad {{\rm{Ad}}}

\def \Spec{{\rm Spec}}

\def \exp {{\rm exp}}

\def \Z {\mathbb{Z}}
\def \R {\mathbb{R}}

\begin{document} 
 
\title{On the Surjectivity of the Power Maps of a Class of Solvable Groups} 
 
\author{S.G. Dani and Arunava Mandal} 
 

 
 \maketitle

\begin{abstract}

 Let  $G$ be a group containing a nilpotent normal subgroup $N$ with central series
 $\{N_j\}$, such that each $N_j/N_{j+1}$ is a $\mathbb{F}$-vector space over a field $\mathbb{F}$ and the action of $G$ on $N_j/N_{j+1}$ induced by the conjugation action is $\mathbb{F}$-linear. For $k\in \N$ we 
describe a necessary and sufficient condition for all elements from any coset $xN$, $x\in G$, to admit  $k$-th roots in $G$, 
 in terms of the action of $x$ on the  quotients $N_j/N_{j+1}.$ This yields in particular a condition for surjectivity 
 of the power maps, generalising various results known in special cases. For $\mathbb{F}$-algebraic groups we also characterise 
the property in terms of centralizers of elements.   For a class of  Lie groups, it is shown that surjectivity of the $k$-th power map, $k\in \N$, implies the same for the restriction of the map to the solvable radical of the group. The results are applied
 in particular to the study of exponentiality of Lie groups.
 
\end{abstract}


\section{Introduction}

Let $G$ be a group. For  $k \in \N$ (a natural number)  we denote by $P_k$ the $k$th power map of $G$,  defined by $P_k(g)=g^k$ for 
all $g\in G$. Inspired by the question of surjectivity of exponential maps of Lie groups there has been interest in  understanding 
conditions for $P_k$ to be surjective. The question was studied by Pralay Chatterjee for various classes groups, beginning with connected 
solvable Lie groups, in \cite{C},   algebraic groups over algebraically closed fields (\cite{C1}), groups of rational points of algebraic groups 
defined over real and $p$-adic fields etc. (see \cite{C2} and references there for details).

 Recently, in \cite{D},  the first named author  extended the study of surjectivity of the exponential  of a solvable Lie 
groups  to describing  conditions for  certain subsets (specifically cosets of certain nilpotent normal Lie subgroups) to be contained in the image of the exponential map; the results were applied in particular to describe conditions under which the radical of an exponential Lie group is exponential, and to generalise a result of Moskowitz and Sacksteder \cite{MS} for complex Lie groups to a large class of Lie groups, on centers of exponential Lie groups. 
 In this paper we study the analogous question for power maps of a large class of solvable groups; see 
 below for the definition of the class and Theorem~\ref{T1} for the statement of the main result.    The groups considered include connected solvable Lie groups, and in this case we deduce from Theorem~\ref{T1} some of the results of \cite{C} concerning the question of surjectivity of the power maps. 
 Our results also yield the corresponding results 
 for the exponential map  proved in \cite{D}, via McCrudden's criterion \cite{Mc} that an element in a Lie group is exponential if an only if it admits roots of all orders. Theorem~\ref{T1} is also applied to deduce surjectivity of the power map of the radical $R$ of a connected Lie group $G$, in analogy with the result in \cite{D} mentioned above, when the corresponding power map of $G$ is surjective, and $G/R$ satisfies a condition, as in~\cite{D}.  

Let $\mathbb{F}$ be a field.
 By a $\mathbb{F}$-{\it nilpotent} group, we mean a nilpotent group $N$ such that if $N=N_0\supset N_1\supset\cdots\supset N_r= \{e\}$
 is the central series of $N$ ($e$ being the identity element of $G$), then each $N_j/N_{j+1}$ is a finite dimensional $\mathbb{F}$-vector space.

Let $N$ be a  $\mathbb{F}$-nilpotent group and  $\{N_j\}$ its central series.   Let $G$ be a group acting on $N$ as a group of automorphisms of $N$.
The $G$-action on  $N$ is said to be  $\mathbb{F}$-{\it linear} if the induced action of $G$ on $N_j/N_{j+1}$ is 
 $\mathbb{F}$-linear for all $j.$

We note that if $G$ is a connected solvable Lie group with nilradical  $N$ and if the latter is simply connected, then $N$ is a $\R$-nilpotent group and the conjugation action of $G$ on $N$ is $\R$-linear. Also, for any field $\F$, if $G$ is the group of $\F$-points of a Zariski-connected  solvable algebraic group and $N$ is the unipotent radical of $G$ then $N$ is $\F$-unipotent and the conjugation action of $G$ on $N$ is $\F$-linear. Starting with these examples one can also construct examples  of non-algebraic groups $G$ with  $\F$-nilpotent normal subgroups $N$ of $G$ such that the conjugation action of $G$ on $N$ is $\F$-linear; we note that the condition holds in particular for any subgroup of $G$ as above containing $N$, in place of $G$ itself, and these include non-algebraic groups. 

In our results $\F$ is allowed to be of positive characteristic. 
In the following, if $\F$ is a field of characteristic $p$, we say that a natural number $k$ is coprime to the characteristic of $\F$ if either $p=0$ or $(k,p)=1$, namely $k$ is not divisible by $p$. For any group $G$ and $k\in \N$ we shall  denote, throughout, by $P_k$ the power map of $G$ defined by $P_k(g)=g^k$ for all $g\in G$, and by 
$P_k(G)$ the image of $P_k$. 

The following is the main technical result of the paper.

\begin{theorem}\label{T1}
{\it Let $G$ be a group and $N$ be a normal subgroup of $G$. Suppose that $N$ is 
 $\mathbb{F}$-nilpotent with respect to a field $\F$, and  that the conjugation action of 
 $G$ on $N$ is $\mathbb{F}$-linear. Let $N=N_0\supset N_1\supset\cdots\supset N_r= \{e\}$
 be the central series of $N$.  Let $A=G/N,$ $x\in G$ and $a=xN\in A.$
 Let $k\in \N$ be coprime to the characteristic of $\F$. Let   $B=\{b\in A\mid b^k=a\}$
 and $B^*$ be the subset consisting of all $b$ in $B$ such that for any $j=1,\dots ,r$ any element of $N_{j-1}/N_{j}$ which is fixed under the 
 action of $a$ is  also fixed under the action of $b$. Then we have the following:
 
i)  for any $b\in B^*$  and $n\in N$, there exists $y\in G$ such that $yN=b$ and $y^k=xn$;
  
 ii) if  $A$ is abelian,  and $xn\in P_k(G)$ for all $n\in N$,  then $B^*$ is non-empty. 
 
 }
\end{theorem}

For groups of $\F$-rational points we deduce the following Corollary;  
in the case when $\F$ is an algebraically closed field of characteristic zero, the  result is contained in \cite{C1}.

\begin{corollary}\label{cor:alg}
Let $G$ be the group of $\F$ points of a solvable algebraic group $\bf G$ defined over $\mathbb{F}$ and let $N$ be the group of $\F$-points of the unipotent radical $R_u({\bf G})$ of $\bf G$.  Let $k$ be coprime to the characteristic of $\F$.  Let $x$ be a semisimple element in $G$. Then  $xn\in P_k(G)$ for all $n\in N$  if and only if there exists $y\in Z(Z_G(x))$ such that $y^k=x.$
  \end{corollary}

Theorem~\ref{T1} yields in particular the following generalisations of certain results proved in \cite{C}, thereby putting them  in a broader 
perspective.  For any Lie subgroup $S$ a Lie group $G$ we denote by $L(S)$ the corresponding Lie subalgebra. For any $X\in L(G)$ and a Lie subgroup $T$ we denote by $Z_T(X)$, the centraliser of $X$ in $T$, namely $\{t\in T\mid \Ad (t)(X)=X\}$.

 \begin{corollary}\label{Cor1.2}
 Let $G$ be a connected solvable Lie group, $N$ be the nilradical of $G$ and $H$ be a Cartan subgroup of $G$. Let $h\in H$ and $k\in \N$. Then $hn \in P_k(G)$ for all $n\in N$ if and only if  there exists 
 $g\in H$ such that $g^k=h$ and $g\in Z_H(X)$ for every $X \in L(N)$ such that $h\in Z_H(X) $. 
In particular,  $P_k:G\to G$ is  surjective if and only if $P_k:Z_H(X)\to Z_H(X)$ is
 surjective for all $X\in L(N).$
\end{corollary}

Via McCrudden's criterion \cite{Mc} recalled above Theorem~\ref{T1} yields  
the following variation of  Theorem~2.2 of \cite{D}. 

\begin{corollary}\label{D-exp}
Let $G$ be a connected solvable Lie group and $N$ be a  connected nilpotent closed normal subgroup of $G$ such that $G/N$ is abelian. Let $A=G/N$ and $N=N_0\supset N_1\supset\cdots\supset N_r= \{e\}$
 be the central series of $N$. Let $x\in G$ and $a=xN/N$. Then $xn$ is exponential in $G$ for all $n\in N$ if and only if 
 there exists a one-parameter subgroup $B$ of $A$ containing $a$ such that for any $j=1,\dots ,r$ any point of $N_{j-1}/N_{j}$ which is fixed by the 
 action of $a$ is  also fixed by the action of $B$.

\end{corollary}

In analogy with the criterion from \cite{D} for exponentiality of radicals we deduce the following criterion for surjectivity of the power maps of radicals. 

 \begin{corollary}\label{A}
 Let $G$ be a connected Lie group such that $P_k:G\rightarrow G$ is
 surjective. Let $R$ be the (solvable) radical of $G$ and  $S=G/R$. Suppose that $S$ has a unipotent 
 one-parameter subgroup $U$ such that $Z_S(U)$ does not contain any element whose order divides $k.$
 Then $P_k:R\rightarrow R$ is surjective.  \end{corollary}
 
The paper is organised as follows. In \S\,2 we prove a preliminary result which is applied in \S\,3 to complete  the proof of Theorem~\ref{T1}. In \S\,4 we discuss the case of power maps of algebraic groups and prove 
Corollary~\ref{cor:alg}. Corollary~\ref{Cor1.2}  is proved  in \S\,5, where 
we discuss some more 
applications of Theorem~\ref{T1} (see Corollaries~\ref{cor5.5} and \ref{cor5.6}) for the power maps of solvable Lie groups. Corollary~\ref{A} on power maps of radicals is proved in \S\,6. 

\section{Results for linear actions on vector spaces}

In this section we prove various results for groups which are semidirect products of cyclic groups with vector spaces, to be used  later to deduce the main theorem.

Let $\F$ be a field and $V$ be a finite-dimensional $\mathbb{F}$-vector space. We denote by $\GL(V)$ the 
group of nonsingular $\F$-linear transformations of $V$. For $\tau \in \GL(V)$ we denote by $F(\tau)$ the set of points fixed by $\tau$, viz. $F(\tau) =\{v\in V\mid \tau v=v\}$. 

We note that if $\tau \in \GL(V)$ and $k\in \N$ are  such that $\tau$ is not unipotent and $\tau^k$ is unipotent
then $F(\tau)$ is a proper subspace of $F(\tau^k)$. Also, for any $\tau \in \GL(V)$ and $k\in \N$ there exists a unique minimal $\tau$-invariant subspace $V'$ such that the (factor) action of  $\tau^k$ on $V/V'$ is trivial while that of $\tau$ has no nonzero fixed point. These results can be proved by a straightforward application of 
the decomposition into generalised eigenspaces, or equivalently the Jordan canonical form of matrices; we omit the details.

\begin{proposition}\label{P1}
 {\it  Let $\F$ and $V$ be as above.  
 Let $G$ be a group with (a copy of) $V$ as a normal subgroup such that $G/V$ is cyclic, and  the conjugation action of $G$ on $V$ is $\mathbb{F}$-linear. Let $\sigma:G\to \GL(V)$ denote the induced action. 
 Let   $k\in \N$ be coprime to the characteristic of $\F$. Let  $g\in G$ and $x=g^k$. Then the following statements hold:   

i) if  $F(\sigma (x))=(0)$ then for every $v\in V$, there exists a unique $w\in V$ such that $xv=wxw^{-1}$. 
 
ii)  if $F(\sigma (x)) =F(\sigma (g))$  then  $xv \in P_k(G)$ for all $v\in V$.
 } 
\end{proposition}

\proof i) Let $I$ denote the identity transformation. The hypothesis implies that  $(\sigma(x)^{-1}-I) \in \GL(V)$.
  Hence for any $v\in V$ there exists $w\in V$ such that $v=(\sigma(x)^{-1}-I)(w)= x^{-1}wxw^{-1}$, in 
  the group structure of $G$; hence we have $xv=wxw^{-1}$.
The uniqueness of $w$ follows from the fact that if $xv=w_1xw_1^{-1}=w_2xw_2^{-1}$, then $w_2^{-1}w_1
\in F(\sigma (x))$ which is given to be trivial.

  ii)  Let $W$ be the largest $\sigma(x)$ invariant subspace of $V$ such that the restriction of $\sigma (x)$ to $W$ is unipotent. Then $W$ is also 
 $\sigma (g)$-invariant, and the hypothesis together with the remark preceding the proposition (applied to $W$ in place of $V$) implies that the restriction of $\sigma (g)$ to $W$ is unipotent.  We now proceed by induction on the dimension of $W$. We note that if $W$  is $0$ dimensional then $F(\sigma (x))=0$ and in this case the desired assertion follows from part~(i). Now consider the general case. Let $U=F(\sigma (x))$ 
 which by the hypothesis is also $F(\sigma (g))$.  Then $U$ is a  normal subgroup of $G$ and $G/U$ contains  $V/U$ as a normal subgroup with cyclic quotient and the action of its generator  on $V/U$ is given by the quotient of the $\sigma (g)$-action on $V/U$. As the restriction of $\sigma (g)$ to $W$ is unipotent, $U$ is of positive dimension and $W/U$ has dimension less than $W$. We note also that $W/U$  is the largest $\sigma (g)$-invariant subspace of $V/U$ on which the $\sigma (x)$-action is unipotent. Hence by the induction hypothesis for any $v\in V$ there exists $\eta \in G$  such that $xvU=\eta^kU$. 
 Thus there exists $u'\in U$ such that $xv=\eta ^ku'$. Since $k$ is coprime to the characteristic of $\F$ there exists 
 $u\in U$ such that $u'=u^k$ (we use the multiplicative notation since $U$ is now being viewed as a subgroup of $G$). Also, as $\sigma (g)$ fixes $u$, the latter is contained in the center of $G$. Hence
 $xv=\eta^ku'=\eta^ku^k=(\eta u)^k$. Thus $xv\in P_k(G)$.   \qed

Now let $V$ be a finite-dimensional $\F$-vector space and let $A$ be an abelian subgroup of $\GL(V)$. Let 
$\frak W$ denote the collection of all  minimal $A$-invariant subspaces $W$ of $V$ such that the factor action of any element of $A$ on $V/W$ is semisimple (diagonalisable over the algebraic closure of $\F$) and can be decomposed into irreducible  components which are isomorphic to each other as $A$-modules; the subspaces from $\frak W$ 
may be arrived at by considering the largest semisimple quotient of the $A$-action on $V$ and decomposing it into isotypical components (putting together irreducible submodules isomorphic to each other). We note that $\frak W$ is a finite collection of proper subspaces of $V$. We shall denote by $X(A,V)$ the subset $\cup_{W\in \frak W} W$. We note that $X(A,V)$ is a proper subset of $V$. 

\begin{proposition}\label{P2}
Let $F$ be a finite-dimensional $\F$ vector space. Let $G$ be a group with $V$ as a normal subgroup such that $G/V$ is abelian. Let $\sigma :G\to \GL(V)$ be the induced action. Let $k\in \N$, $g\in G$ and $x=g^k$. 
 Let  $A=G/V$, $a=xV$, $b=gV\in  A$.  Suppose that $F(a)\neq F(b)$.  Let $v\in V$ be such that  there exists  $y\in G$ satisfying  $yV=b$ and $y^k=xv$. Then $v\in X(\sigma (A),V)$. 

\end{proposition}

\proof Using the Jordan canonical form we see that there exists a  unique minimal $a$-invariant subspace, say $U$, of $V$ such that the action of $a$ on $V/U$ is trivial while the $b$-action on $V/U$ has no nonzero fixed point.  The condition in the hypothesis  implies that $U$ is a proper subspace of $V$. 
As $A$ is abelian it follows that $U$ is $A$-invariant. Moreover $U$ is contained in $X(\sigma (A),V)$ as we can find a 
subspace $W\in \frak W$ containing it. We show that any $v$ satisfying the condition in the hypothesis is contained in $U$. Let such a $v$ be given and $y\in G$ be such that $yV=b$ and $y^k=xv$. Then there exists $w\in V$ such that $y=gw$
and we have  $xv=y^k= (gw)^k= (gw) \cdots (gw) =
 g^k(g^{-(k-1)}wg^{(k-1)}) \cdots (g^{-1}wg)w=g^k\theta w=x\theta w$,
 where $\theta = \sigma (g)^{-(k-1)}+ \cdots +\sigma (g)^{-1} +I$, and hence $v=\theta w$. 
We  note that $V$ can be  decomposed  as $U\oplus V'$, where  $V'$ is a $\sigma (g)$-invariant subspace such that any eigenvalue of the restriction of $\sigma (g)$  to $V'$ is contained in  $\{\lambda \in \C \mid \lambda^k=1,\lambda \neq 1\}$. 
Also, $(\sigma (g)^{-1}-I)\theta  =\sigma(g)^{-k}-I=\sigma (x)^{-1}-I=0$, the zero transformation. Since $\sigma (g)$ has no nonzero fixed point on $V'$,  the restriction of $(\sigma (g)^{-1}-I)$ to $V'$ is invertible and   hence we have $\theta (v')=0$ for all $v'\in V'$.  Let $u\in U$ and $v'\in V'$ be 
such that $w=u+v'$. Then we have $v=\theta w=\theta u + \theta v' =\theta u\in U$, as desired. 
\qed

\section {Proof of Theorem \ref{T1} } 
In this section we shall deduce Theorem~\ref{T1} from the results of \S\,2. In the sequel, given a group $H$,
for $h\in H$ we denote by $\langle h\rangle$ the cyclic subgroup generated by $h$. 

\noindent {\it Proof of (i)} In proving this part without loss of generality we may assume that  $G$ is the product of 
$\langle g\rangle$ with $N$, where $g \in G$ is such that $gN=b$, with $b$ a given element from the subset $B^*$; while the assumption is not crucial it makes the proof more transparent.  
   Let $N=N_0\supset N_1\supset\cdots\supset N_r= \{e\}$ be the central series of $N.$ We proceed by induction on $r$. For $r=1$ the desired statement is immediate from Proposition~\ref{P1}(ii). Now consider 
   the general case. Let $n\in N$ be given. Now $G/N_{r-1}$ is a product of  
   $\langle gN_{r-1}\rangle$ and  $N/N_{r-1}$ in $G/N_{r-1}$, with $g\in G$ as above, and   
   the condition in the hypothesis of the theorem is  satisfied for $G/N_{r-1}$ (with $N/N_{r-1}$, $a_1=g_1^kN_{r-1}$ and $b_1=g_1N_{r-1}$ in place of $N$ $a$ and $b$ respectively); we note that the  
   corresponding actions on $N_j/N_{j+1}$ coincide those of $a$ and $b$ respectively. Since 
   $G/N_{r-1}$ is a 
   $\F$-nilpotent group of  length less than $r$, by the induction hypothesis there exists
   $g_1\in G$ such that $g_1N=b$ and $xnN_{r-1}=g_1^kN_{r-1}$.
   Hence there exists $v\in N_{r-1}$ such that $xn=g_1^kv$.
   Now consider the product, say $G_1$, of the subgroups $\langle g_1\rangle$ and 
   $N_{r-1}$ in $G$; we note that $N_{r-1}$ is   a $\F$-vector space, and the condition as in Proposition~\ref{P1}(ii) is satisfied.  Hence we get that there exists  $y\in G_1$ such that 
    $g_1^kv=y^k$;  
    thus we have   $y\in G$ such that $yN=b$ and $xn=y^k$, and so $xn\in P_k(G)$, which proves (i). 
   
   (ii)  For any subset  $S$  of $B$ we denote by $E_S$ the set of $n$ in 
$N$ for which  there exists $y\in G$ such that $yN\in S$ and $y^k=xn$. We shall show that 
$E_S=N$ if and only if 
$S\cap B^*$ is nonempty; when the condition in (ii) holds, for $S=B$ we have $E_S=N$ and this implies 
assertion (ii) in the theorem.  We proceed by induction on $r$, the length of the central series. For $r=1$
the assertion follows immediately from Propositions~\ref{P1} and~\ref{P2}. Now consider the general case. 
Let  $S$ be  a subset of $B$. If   $S\cap B^*$ is nonempty, then assertion~(i) of the theorem, proved 
above, shows that $E_S=N$. Now suppose that  $E_S=N$. 
Let $V=N_{r-1}$, $G'=G/V$ and 
$B'$ be the subset of $B$ consisting of $b$ such that for all $j=1,\dots , r-1$ any point of $N_{j-1}/N_j$ 
which is fixed by $a$ is also fixed by $b$.  Since $E_S=N$, for all 
$n\in N$ there exists $y\in G$ such that $yN \in S$ and $y^kV=xnV$, and hence the induction hypothesis 
implies that $S\cap B'$ is nonempty. Also, applying the induction hypothesis to  the set $S\backslash B'$ 
we get that there exists $n_0\in N$ for which there does not exist any $y\in G$ such that $yN\in S\backslash B'$ 
and $y^kV=xn_0V$. Now consider $xn_0v$, with $v\in V$. Since $E_S=N$ there exists  $y\in G$ 
(depending on $v$) such that $yN\in S$ and 
$y^k=xn_0v$, and the preceding conclusion shows that in fact $yN\in S\cap B'$. To show that $S\cap B^*$ is 
nonempty we have to show that there exists $b\in S\cap B'$ such that every fixed point of $a$ on $V$ is fixed by $b$. Suppose that this is not true; thus $F(a)\neq F(b)$ for all $b\in S\cap B'$, in the notation as before, with respect to the action on $V$ as above. Since the action of $a$ on $V$ is the same as the (conjugation) action of $xn_0$, 
Proposition~\ref{P2} implies that any $v\in V$ for which there exists a $y \in G$ such that $yN\in S\cap B'$ and
$y^k=xn_0v$ is contained in $X(\sigma (A),V)$, in the notation as before,  $\sigma$ being the conjugation action. This contradicts the conclusion as above, 
since $X(\sigma(A),V)$ is a proper subset of $V$. This shows that $S\cap B^*$ is nonempty, and completes the proof of the theorem. \qed

 \section{Power maps of solvable algebraic groups}
 
 In this section we discuss power maps of solvable algebraic groups and prove Corollary~\ref{cor:alg}.
 Let $G$ be the group of $\F$ points of a solvable algebraic group $\bf G$ defined over $\mathbb{F}.$ By Levi decomposition $G=T\cdot N$ (semidirect product), where  $T$ consists of the group of $\F$-points of a maximal torus in $\bf G$ defined over $\mathbb{F}$, and $N$ is the group of $\F$-points of the unipotent radical $R_u({\bf G})$ of $\bf G$.   
 
 \begin{remark}\label{centr}
 {\rm 
 Let $N=N_0\supset N_1 \supset \cdots \supset N_r =\{e\}$ the central series of $N$, with $e$ the identity element. For $j=0,\dots, r-1$ let $p_j:N_{j}\to N_{j}/N_{j+1}$ be the quotient homomorphism. 
 For $j=0,1, \dots , r-1$ let $\frak M_j$ be the $\Ad (T)$-invariant subspace of the Lie subalgebra of $N_j$ 
 complementary to the Lie subalgebra of $N_{j+1}$, and let $\frak P_j$ be the  
 collection of (algebraic) one-parameter subgroups $\rho$ corresponding to one-dimensional subspaces of $\frak M_j$. Let $\frak P =\cup_j\frak P_j$. By considering the decomposition of the Lie algebra 
 of $N$ with respect to the action of $T$ it can be seen that, for any $t\in T$,  $Z_N(t)$ is generated by the collection of one-parameter subgroups $\rho \in \frak P_j$, $j=0,\dots,r-1$,  such that any $v\in p_j(\rho)$  is fixed under the action of $t$ on $N_j/N_{j+1}$.  Moreover, any one-dimensional subspace of $N_j/N_{j+1}$ 
 which is pointwise fixed under the action of $t$ is of the form $p_j(\rho)$ for some  $\rho \in \frak P_j$ centralised by $t$. }
 \end{remark}
 
   \bigskip
\noindent{\it Proof of Corollary~\ref{cor:alg}}:  Let $T$ be the subgroup as above. Then every semisimple element has a conjugate in $T$ and hence it suffices to prove the assertion in the corollary for $x$ in $T$. Let $x\in T$ be given and $a=xN\in A=G/N$. 
 We note  that for any $t\in T$, $Z_G(t)=TZ_N(t)$. Suppose there exists $y\in Z(Z_G(x))$ such that $y^k=x$. 
 Thus $y$ has the form $sn$, $s\in T$,
$n\in Z_N(x)$, with $sn$, and hence $n$, commuting with all elements of $T$. Hence $x=y^k=s^kn^k$, which 
implies that $s^k=x$ and $n^k=e$. Since $k$ is coprime to the characteristic of $\F$ we get that $n=e$ and hence $y=s\in T$. Since any $\rho \in \frak P_j$, $0\leq j\leq r-1$, which is centralised by $x$ is also centralised by $y$, by Remark~\ref{centr} any $v\in N_j/N_{j+1}$, $0\leq j\leq r-1$, is fixed under the action of $y$. Thus for $a=xN$ as above there exists $b=yN$ for which the condition of Theorem~\ref{T1} is satisfied. 
Hence by the theorem $xn\in P_k(G)$ for all $n\in N$. 

Conversely suppose that $xn\in P_k(G)$ for all $n\in N$. Then by Theorem~\ref{T1} there exists $b\in G/N$ 
such that $b^k=a$ and any $v\in N_{j-1}/N_j$, $j=1, \dots r$  which is fixed under the action of $a$ is 
also fixed under the action of $b$. Let $g\in G$ be such that $gN=b$. Let $g=yn$ with $y\in T$ and $n\in N$. 
Then as $b^k=a$ we get $y^k=x$. Also, the actions of $b$ and $y$ on any $N_{j}/N_{j+1}$, $j=0, \dots r-1$ 
coincide. Thus  any $v\in N_{j-1}/N_j$, $j=1, \dots r$  which is fixed under the action of $x$ is 
also fixed under the action of $y$. From Remark~\ref{centr} get that every $\rho \in \frak P_j$, 
$0\leq j \leq r-1$, which is centralised by $x$ is centralised by $y$, and in turn that $Z_N(x)=Z_N(y)$.  Since $Z_G(x)=TZ_N(x)$ we get that  $Z_G(x)=Z_G(y)$, which shows that  $y\in Z(Z_G(x))$. \qed

  \section{Power maps of solvable Lie groups}
  
 Now let $G$ be a  connected solvable (real) Lie group and $N$ be a simply connected nilpotent Lie group of $G$ such that $G/N$ is abelian.
 Let $N=N_0\supset N_1\supset\cdots\supset N_{r-1}\supset N_{r}= \{e\}$ be the central series of $N$. Note that as $N$ is simply
 connected, $N_j/N_{j+1}$ are real vector space and therefore $G$ action on $N$ is $\mathbb{R}$-linear.
  In this case Theorem~\ref{T1} implies the following. 
   
 \begin{corollary}\label{B}
  Let $G$ be a connected Lie group and let $N$ be a simply connected nilpotent closed normal subgroup of $G$ such that $G/N$ is abelian. 
   Let $A=G/N$, $x\in G$ and $xN=a\in A.$ Let  $N=N_0\supset N_1\supset\cdots \supset N_{r}= \{e\}$ 
  be the central series of $N$. Let $k\in\mathbb{N}.$ Then $xn\in P_k(G)$ for all $n\in N$ if and only if  there exists $b\in A$ with $b^k=a$ such that any fixed point of $a$ in $N_j/N_{j+1}$ is also fixed by $b.$
\end{corollary}

Corollary~\ref{B} may be reformulated as follows, in a form comparable to the result for the exponential maps 
proved in  \cite{D}, Theorem~2.2. 
 If $M$ and $M'$ are closed connected normal subgroup of $G$ contained in $N$ such that  $M\supset M'$ and 
the $G$ action  on $M/M'$ is irreducible, then the pair $(M,M')$  is called an irreducible 
subquotient of $N$ (with respect to the $G$-action) (cf.~\cite{D}). Using Jordan canonical form it can be seen that condition~(i) as in Corollary~\ref{B} is 
equivalent to the condition that there exists $b\in A$ such that for any irreducible subquotient $(M,M')$ for which  
the action of $a$ on $M/M'$ is trivial the action of $b$ on $M/M'$ is also trivial. Hence we get the following:

 \begin{corollary}{\rm (\cite{D})}\label{C1}
  Let the notation be as in Corollary~\ref{B}. Then $xn \in P_k(G)$ for all $n\in N$ if and only if there exists $b\in A$ such that
for any irreducible subquotient $(M, M')$, where $M,M'$ are closed connected normal subgroups of $G$ contained in $N$, if  the action of $a$ on $M/M'$ is trivial, then the action of $b$  on $M/M'$ is trivial.   
 \end{corollary}
 
 We recall also 
that an element $g$ in a connected Lie group $G$ is exponential if and only if it is contained in $P_k(G)$ 
for all $k$ \cite{Mc}. In view of this, Corollary~\ref{B} implies Corollary~\ref{D-exp}, which is a variation of the characterisation in  \cite{D}. 

 
We now describe some more   applications of our results, after recalling some structural aspects of solvable Lie groups.  
Let $G$ be a connected solvable Lie group and $N$ be nilradical of $G.$ 
We denote by $N=N_0\supset N_1 \supset \cdots \supset N_r=\{e\}$ the central series of $N$, with $e$ 
the identity element. 
Let $H$ be a Cartan subgroup of $G.$ We note that $H$ is a connected  nilpotent subgroup and $G$ can be written as $G=HN.$  
  We have a weight space decomposition of $L(N)$ (Lie algebra of $N$, as before) with respect to the adjoint action of $H$, as 
 $L(N)=\oplus_{s\in\bigtriangleup} L(N)_s$, where $\bigtriangleup$ is a set of weights, and 
$L(N)_s$, $s\in \Delta$,  are $\Ad(H)$-invariant subspaces of $L(N)$. For $x\in H$, the restriction of $\Ad(x)$ to $L(N)_s$ has either only one real eigenvalue or a pair of complex numbers as eigenvalues; in either case we shall denote  the eigenvalue(s) by $\lambda(x,s)$ and  $\bar\lambda(x,s).$ 

We note the following observations; the proofs are straightforward and will be omitted. 

\begin{remark}\label{L}
{\rm Let  $x, y\in  H$ be such that  $y^k = x$.  Then  for all $j=0, \dots , r-1$, every  $v\in N_j/N_{j+1}$ which is fixed 
under the action of $x$ is also fixed under the action of $y$ if and only if  for every $X \in  L(N)$ such that $\Ad(x)X = X$  
we also have $\Ad(y)X = X$.}
\end{remark}

\begin{remark}\label{sconn}
{\rm The nilradical $N$ has a unique maximal compact subgroup $C$; moreover $C$ is 
a connected subgroup contained in the center of $G$ and $N/C$ is simply connected. Consequently,
 for $g\in G$, for any $k\in \N$, $g\in P_k(G)$ if and only if $gC\in P_k(G/C)$.  }
\end{remark}

\medskip
\noindent{\it Proof of Corollary~\ref{Cor1.2}}: In view of Remark~\ref{sconn} we may assume $N$ to be simply connected. Thus $N$ is a $\R$-nilpotent group and the $G$-action on $N$ is $\R$-linear. Let $h\in H$ be given. 
Suppose that there exists $g\in H$ such that $g^k=h$ and $g\in Z_H(X)$ for all $X\in L(N)$ such that $h\in Z_H(X)$. Then Theorem~\ref{T1} together with Remark~\ref{L} implies that $hn\in P_k(G)$ for all $n\in N$. 
Conversely suppose that 
  $hn\in P_k(G)$ for all $n\in N$. Let $a=hN$. By Theorem~\ref{T1} there exists $b\in A=G/N$ such that any 
  $v\in N_j/N_{j+1}$, $0\leq j\leq r-1$, which is fixed by $a$ is also fixed by $b$. Since $G=HN$ and $H$ is a connected nilpotent Lie group, there exists $g\in H$ such that $gN=b$ and $g^k=h$. By Remark~\ref{L}  we also have   $g\in Z_H(X)$ for all $X\in L(N)$ such that $h\in Z_H(X)$, since on each $N_j/N_{j+1}$ the action 
  of $g$ is the same as that of $b$. This proves the first assertion in the Corollary. 
  
  Now suppose that $P_k:G\to G$ is surjective. Let $X\in L(N)$ be given.  For any $h\in Z_H(X)$  we have
  $hn\in P_k(G)$ for all $n\in N$, and hence by the above there exists $g\in Z_H(X)$ such that $g^k=h$, 
  which shows that $P_k:Z_H(X)\to Z_H(X)$ is surjective. Conversely suppose that $P_k:Z_H(X)\to Z_H(X)$ is surjective for all $X\in L(N)$. Since $G=HN$ it suffices to show that $hn\in P_k(G)$ for all $h\in H$ and $n\in N$. Let $h\in H$ be given, and let $\Delta'=\{s\in \Delta \mid \lambda(h,s)=1\}.$  
 For all $s\in\Delta'$ let $X_s\in L(N)_s$ be such that 
 $X_s\neq 0$, and  let $X=\sum_{s\in \Delta'}X_s$. Then $h\in Z_H(X)$ and hence by the first part there exists $g\in Z_H(X)$ such that $g^k=h$. Then  $g\in H$ and,  since $\Ad (g)(X)=X$, we have $\Ad(g)X_s=X_s$ for all $X_s$, for all $s\in\Delta',$ and hence   
 $\lambda(g,s)=1$ for all $s\in \Delta'.$ This in turn implies that  $g\in Z_H(Y) $ for all $Y \in  L(N)$ such that $h\in Z_H(Y)$. The first part of the Corollary proved above now implies that $hn \in P_k(G) $ for all $n\in N$, as sought to be shown. \qed

\medskip
 For $g\in G$ we denote by $\Spec (g)$ the set of all (complex) eigenvalues of $\Ad (g)$. 
 An element $g\in G$ is said to be $P_k$-{\it regular} if $\Spec (g)\cap\{\lambda| \lambda^k=1, \lambda\neq 1\}=\emptyset.$

 \begin{corollary}\label{cor5.5}
  Let $G$ be a connected solvable Lie group and $k\in \N$. For any  $x\in P_k(G)$  there exists a $P_k$-regular element $y$ in $G$ such that $y^k=x$. If $P_k$ is surjective then $P_k:Z(G)\rightarrow Z(G)$ is surjective.
 \end{corollary}
\begin{proof}
$(i)$   We note that in view of  Remark \ref{sconn} we may assume that  $N$   simply connected. 
First suppose  $x\in H$, a Cartan subgroup.  
  Since  $xn\in P_k(G)$ for all $n\in N$
 by Corollary~\ref{Cor1.2} there exists $y\in H$ such that $y^k=x$ and $\Ad(y)X = X$ for every $X \in  L(N)$ such that $\Ad(x)X = X$. The latter condition implies in particular that $y$ is $P_k$-regular. 
   
 Now consider any  $x\in G$, say $x=hu$, $h\in H$ and $u\in N$.  Then there exists  a $P_k$-regular element 
 $g\in H$ such that $g^k=h$. It is known that for a $P_k$-regular element $g$ and $u\in N$ there exists 
 a $w\in N$ such that $g^ku=(gw)^k$ (see \cite{C}, Proposition~3.5). Thus $x=hu=g^ku=(gw)^k$,  we get $y=gw$ as a $P_k$-regular element such  that $y^k=x$. 
 
 Now suppose that $P_k$ is surjective and let $x\in Z(G)$ be given. By the above assertions there exists a $P_k$-regular element $y$ such that $y^k=x$. Since $\Ad (x)$ is trivial, 
 $y$ being a $P_k$-regular element with $y^k=x$ implies that $\Ad(y)$ is trivial, namely $y\in Z(G)$. Hence $P_k:Z(G)\to Z(G)$ is surjective. 
\end{proof}

\begin{corollary}\label{cor5.6}
  Let $G$ be a simply connected solvable Lie group, $x\in G$  and $k\geq 2$.  Then the following are equivalent:   
  i) $xN \subset P_k(G)$; \ ii) there exists $g\in G$ such that $g^k=x$ and $\Spec\,(g)\cap \{\lambda \in \mathbb{C}|\;|\lambda|=1, \lambda\neq 1\}=\emptyset$; and iii) $\Spec\,(g)\cap \{\lambda \in \mathbb{C}|\;|\lambda|=1, \lambda\neq 1\}=\emptyset$, for all  $g\in G$ such that $g^k=x$. 

\end{corollary}

\begin{proof} That statement (i) implies (ii) follows  immediately from Corollary~\ref{Cor1.2} and the fact that if $g=hu$, where $h\in H$, a Cartan subgroup, and $u\in N$ then ${\Spec} \, (g)=\Spec \,(h)$; this part does not involve $G$ being
simply connected. 

To prove the other two assertions we first note the following. Let $A=G/N$ and $a=xN$. Since $G$ is simply connected it follows that $A$ is a vector space. Hence $a$ has a unique $k$\,th root $b$ in $A$. Thus for any $g\in G$ such that $g^k=x$ we have 
$gN=b$. This firstly shows that (ii) implies~(iii), since $\Spec\, (g)$ is determined $b$. 
Finally suppose  (iii) holds and let $g\in G$ be such that $gN=b$, the unique $k$\,th root of $a$ in $A$.  If there exists $v\in N_j/N_{j+1}$, where $N=N_0\supset \cdots \supset N_r=\{e\}$ is the central series of $N$, which is fixed by the action of $a$ but not by that of $b$, then  $\Ad (g)$ would have an eigenvalue $\lambda$ 
such that $\lambda \neq 1$ and $\lambda^k=1$. Since this is ruled out, 
by Theorem~\ref{T1} we get that $xN\in P_k(G)$, thus proving~(i).  
 \end{proof}

 \section{Surjectivity of the power maps of the radicals}
 
 In this section, we consider power maps of radicals and prove Corollary~\ref{A}, which is the analogue  of Theorem~1.2 of \cite{D}, in the present context.  We note some preliminary results before going over to the proof of Corollary~\ref{A}. 
 
 \begin{lemma}\label{AL1}
  Let $k\in\mathbb{N}.$ Let $U$ be a  one-parameter subgroup of $ S$ such that $Z_S(U)$ does not contain any element whose order
  divides $k.$ Let $u$ be a nontrivial element in $U$ and $v\in Z_S(U)$ be such that  $v^k=u$.  Then $v\in U$.
 \end{lemma}
\begin{proof}
 As $U$ is a one-parameter subgroup there exists $w\in U$ such that $w^k=u$. Since  $v\in Z_S(U)$ 
 we have  $(vw^{-1})^k=v^kw^{-k}=u u^{-1}=e,$ the identity element. Thus $vw^{-1}\in Z_S(u)$ and its order divides $k$. Hence  by hypothesis $vw^{-1}=e$, and  so  $v=w \in U$. 
\end{proof}

\begin{proposition}\label{P5}
Let $G$ be a connected Lie group, $R$ be the radical of $G$, and  $S=G/R$. Let $k\in \N$ be such that $P_k:G\to G$ is surjective. Let  $U$ be  a unipotent one-parameter subgroup of $S$ such that $Z_S(U)$ does not contain any nontrivial element whose order divides $k$. 
 Let $H$ be the closed subgroup of $G$ containing $R$ such that $H/R=U.$ 
 Let $h\in H$ be an element  not contained in $R.$ Then $h\in P_k(H).$
\end{proposition}

 \begin{proof}
  As $P_k$ is surjective there exists $g\in G$ such that $g^k=h.$ Let $v=gR\in S$ and $u=hR \in U$.  Since $v$ commutes with $u$ and $U$ is a unipotent one-parameter subgroup it follows that $v$ commutes with all elements of $U$, viz. $v\in Z_S(U)$. As $h\notin R$, $u$ is a nontrivial element of $U$.  Also, since 
  $g^k=h$ we have $v^k=u$, and hence by  Lemma \ref{AL1} we get $v\in U$. Therefore $g\in H$ and in turn 
 $h\in P_k(H).$
 \end{proof}

\bigskip
 \noindent  {\it Proof of Corollary~\ref{A}}:
 Let $x\in R$ be given.   Let $H$ be as above. We note that since $H/R$ is isomorphic to $\R$, to prove that 
 $x\in P_k(R)$ it suffices to prove that $x\in P_k(H)$.  Also, since $N$ is the nilradical of $G$, $G/N$ is 
 reductive, and hence $H/N$ is abelian. 
 Now let $\{u_t\}$ be the one-parameter subgroup of $G$  such that $U=\{u_tR\}$. Let $u=u_1$ and $y=ux\in H.$ Let $A=H/N$ and $a=xN$, $a'=yN\in A.$
  We note that by Proposition \ref{P5}, $yn\in P_k(H)$ for all $n\in N.$   Hence  by Corollary~\ref{C1}  there exists
  $b'\in A$ such that $b'^k=a'$ and if the $a'$-action on  an irreducible subquotient $(M,M')$ of $N$, with respect to the $H$-action (see \S\,4 for definition) is trivial then the $b'$-action is also trivial. Let $b=(u_{\frac1k}N)^{-1}b'\in A$; then we have $b^k=a$.  
  We note that the action of  $U$ on the Lie algebra of $N$ is unipotent and by the
  irreducibility condition this 
  implies that for any $t$ the action of $u_tN$ on $M/M'$ is trivial.  Hence the action of $a$ on $M/M'$ coincides 
  with the action of $a'$, and similarly the action of $b$ coincides with that of $b'$. Thus we see that 
  $b^k=a$ and if  the action of $a$ on an irreducible subquotient $M/M'$ is trivial then so is the action of $b$. 
  Hence by Corollary~\ref{C1} $xn \in P_k(H)$ for all $n$. Thus  $x\in P_k(H)$ and as noted above it follows 
  that $x\in P_k(R)$.   \qed

\medskip
It is noted in \cite{D} that if $S$ is a complex semisimple Lie group of the group of $\R$-points of a quasi-split semisimple algebraic group defined over $\R$ then it contains a unipotent one-parameter subgroup $U$ such that $Z_S(U)$ does not contain any compact subgroup of positive dimension. In this case the set of  primes dividing the orders of elements of $Z_S(U)$ is finite, say $F$, and for $k\in \N$ which is not divisible by any $p$ in $F$ the one-parameter subgroup $U$ satisfies the condition as in the hypothesis of Corollary~\ref{A}.

\vskip8mm

\begin{flushleft}
S.G. Dani and Arunava Mandal\\
Department of Mathematics\\
Indian Institute of Technology Bombay\\
Powai, Mumbai 400076\\
India

\medskip
E-mail: {\tt sdani@math.iitb.ac.in} and {\tt  amandal@math.iitb.ac.in} 
\end{flushleft}

\end{document}